\documentclass{amsart}

\usepackage{amssymb}

\usepackage{graphicx}
\usepackage{hyperref}
\usepackage{xcolor}
\usepackage{amsmath}

\newtheorem{theorem}{Theorem}[section]
\newtheorem{lemma}[theorem]{Lemma}

\theoremstyle{definition}

\newtheorem{proposition}[theorem]{Proposition}

\newtheorem{Problem}{Problem}
\newtheorem{remark}{Remark}

\theoremstyle{remark}
\usepackage{setspace}
\usepackage{mathtools}
\numberwithin{equation}{section}

\begin{document}

\title[Subset Grunwald-Wang Theorem]{A Generalization of the Grunwald-Wang Theorem for $n^{th}$ Powers}


%    author one information
\author{Bhawesh Mishra}
\address{Dunn Hall 384, Department of Mathematical Sciences, The University of Memphis}
%\curraddr{}
\email{bmishra1@memphis.edu}
\urladdr{https://bhaweshmishra.com/}
\thanks{}

%    \subjclass is required.
\date{\today}
\dedicatory{}

%    "Communicated by" -- provide editor's name; required.
\commby{}

%    Abstract is required.
\begin{abstract}
Let $n$ be a natural number greater than $2$ and $q$ be the smallest prime dividing $n$. We show that a finite subset $A$ of rationals, of cardinality at most $q$, contains a $n^{th}$ power in $\mathbb{Q}_{p}$ for almost every prime $p$ if and only if $A$ contains a perfect $n^{th}$ power, barring some exceptions when $n$ is even. This generalizes the Grunwald-Wang theorem for $n^{th}$ powers, from one rational number to finite subsets of rational numbers. We also show that the upper bound $q$ in this generalization is optimal for every $n$. 
\end{abstract}

\maketitle

\section{Introduction}
In this note, the phrase \textit{almost every prime} means all but finitely many primes. A rational (resp. an integer) $a$ will be called \textit{perfect $n^{th}$ power} when $a = r^{n}$ for some rational (resp. some integer) $r$. 

Using elementary methods, one can show that a rational number $a$ is a square in $\mathbb{Q}_{p}$ for almost every prime $p$ if and only if $a$ is a perfect square. Following is a generalization of this statement from square to an arbitrary power, a version of which for general number fields can be found in \cite[Chapter X, Theorem 1]{ArtinTate}.

\begin{proposition}\textbf{\big(Grunwald-Wang Theorem for $n^{th}$ Powers\big)}
Let $a \in\mathbb{Q}$, $n \in\mathbb{N}$. Then, $a$ is a $n^{th}$ power in $\mathbb{Q}_{p}$ for almost every prime $p$ if and only if one of the following is true:
\begin{enumerate}
    \item $a$ is a $n^{th}$ power in $\mathbb{Q}$. 

    \item $8 \mid n$ and $a = 2^{\frac{n}{2}} \cdot b^{n}$ for some $b \in\mathbb{Q}$. 
\end{enumerate}
\label{GrunWang}
\end{proposition}
Roughly speaking, the Grunwald-Wang theorem states that given any finite subset $S$ of primes and local continuous characters $\chi_{p}$ of $\mathbb{Q}_{p}$ for every $p\in S$, there exists a global character $\chi$ of $\mathbb{Q}$ that extends $\chi_{p}$ for every $p\in S$. We refer the readers to \cite[Chapter X, Theorem 5]{ArtinTate} for a precise statement of the Grunwald-Wang theorem. 

Proposition \ref{GrunWang} and the Grunwald-Wang theorem are closely related, and they often go by the same name. However, in this article, we will call Proposition \ref{GrunWang}, the Grunwald-Wang theorem for $n^{th}$ powers. The proof of the Grunwald-Wang theorem, that the author is aware of, uses Galois cohomology to first establish the Proposition \ref{GrunWang}, which is equivalent to the fact that the natural map 
\begin{equation}
\mathbb{Q}^{\times}/\big(\mathbb{Q}^{\times}\big)^{n} \longrightarrow \prod_{p \notin S} \Bigg(\mathbb{Q}_{p}^{\times}/\big(\mathbb{Q}_{p}^{\times}\big)^{n}\Bigg), \label{injective}    
\end{equation}
induced by the embedding $\mathbb{Q}^{\times} \xhookrightarrow{} \mathbb{Q}_{p}^{\times}$, is injective, except in the case $8 \mid n$, when the kernel of the map \[\mathbb{Q}^{\times} \longrightarrow \prod_{p \notin S} \Bigg(\mathbb{Q}_{p}^{\times}/\big(\mathbb{Q}_{p}^{\times}\big)^{n}\Bigg)\] is equal to $\big(\mathbb{Q}^{\times}\big)^{n} \bigcup 2^{n/2}\big(\mathbb{Q}^{\times}\big)^{n}$. 

Let $K$ be a local or global field, $G_{K}$ be the absolute Galois group of $K$, $\mu_{n}$ be the group of $n^{th}$ roots of unity and let $H^{1}(G_{K}, \mu_{n})$ denote the first continuous cohomology with respect to the profinite topology on $G_{K}$. In other words, $H^{1}(G_{K}, \mu_{n})$ is the group of all the continuous group homomorphisms from $G_{K}$ to $\mu_{n}$. Using the fact that for any local or global field $K$, $H^{1}(G_{K}, \mu_{n}) \cong K^{\times}/\big(K^{\times}\big)^{n}$ (for instance, see \cite[pp. 344]{NeukrichCohom}), injectivity of \eqref{injective} is equivalent to the injectivity of \[H^{1}(G_{\mathbb{Q}}, \mu_{n}) \longrightarrow \prod_{p \notin S}H^{1}(G_{\mathbb{Q}_{p}}, \mu_{n}), \] which elucidates the Galois-cohomological nature of Proposition \ref{GrunWang}. Then the Grunwald-Wang theorem is established from Proposition \ref{GrunWang} using class-field theory. 

Grunwald first published a proof of the Grunwald-Wang theorem and Whaples published another proof, in \cite{Grunwald} and \cite{Whaples} respectively. However, both of them had overlooked an exceptional case, which Wang first demonstrated in 1948, and he ultimately outlined the exact conditions when the original statement has exceptions \cite{Wang1,Wang2}. Hence, the correct result is aptly called the Grunwald-Wang theorem. A short but fascinating account of this whole story can be found in \cite[Section 5.4, pp. 27-29]{Roquette}. Moreover, the Grunwald-Wang theorem has been generalized in various directions: for instance, when $\mathbb{Z}/n\mathbb{Z}$ is replaced with other finite modules \cite[VIII.6.10]{NeukrichCohom}, a version of it for abelian varieties \cite{Creutz}, and an effective Grunwald-Wang Theorem \cite{Effective}. 

In this article, we generalize the Grunwald-Wang theorem for $n^{th}$ powers from one rational to a subset of rationals, using surprisingly modest methods. We start our discussion with an illustrative example.

\subsection{Motivation}
In contrast to dealing with one rational number that is $n^{th}$ power in $\mathbb{Q}_{p}$ for almost every prime $p$, one can consider sets $A$ that contain a $n^{th}$ power in $\mathbb{Q}_{p}$ for almost every prime $p$. This scenario is of interest because there exist finite subsets $A$ of rationals such that none of the elements in $A$ satisfy $(1)$ and $(2)$ of Proposition \ref{GrunWang}, and yet $A$ contains a $n^{th}$ power in $\mathbb{Q}_{p}$ for almost every prime. For instance, we will demonstrate examples of such sets for $n = 2, 3$ and 6. Let $a$ and $b$ be distinct non-zero integers. Define 
\begin{equation}
    A = \big\{a, b, ab, ab^{2}\big\}. \label{intro1}
\end{equation} For any prime $p\neq 3$ and $p \nmid ab$, we have three cases:
\begin{enumerate}
    \item $a$ or $b$ is a cubic residue modulo $p$. 

    \item $a$ and $b$ lie in the same non-trivial coset in $\big(\mathbb{F}_{p}\big)^{\times}/\big(\mathbb{F}_{p}^{\times}\big)^{3}$. In this case, $ab^{2}$ is a cubic residue modulo $p$. 

    \item $a$ and $b$ lie in different non-trivial cosets in$\big(\mathbb{F}_{p}\big)^{\times}/\big(\mathbb{F}_{p}^{\times}\big)^{3}$. In this case, $ab$ is a cubic residue modulo $p$. 
\end{enumerate}
In all the three cases, $A$ contains a cubic residue modulo $p$.  Then, the Hensel's lemma implies that the set $A$ contains a $3^{rd}$ power in $\mathbb{Q}_{p}$ for almost every prime $p$. 

Moreover, the set $A^{(2)} := \{a^{2}, b^{2}, a^{2}b^{2}, a^{4}b^{2}\}$ contains a $6^{th}$ power modulo $p$ for primes $p \nmid 3ab$. Now, the Hensel's lemma implies that the set $A^{(2)}$ contains a $6^{th}$ power in $\mathbb{Q}_{p}$ for almost every prime $p$. If we choose $a$ and $b$ to be distinct and cube-free, then $A^{(2)}$ does not contain a perfect $6^{th}$ power. This shows that there exists infinitely many subsets of $\mathbb{Q}^{\times}/\big(\mathbb{Q}^{\times}\big)^{3}$ of cardinality four that do not contain any $6^{th}$ power, yet contain a $6^{th}$ power in $\mathbb{Q}_{p}$ for almost every prime $p$.

One may wonder if there exist a set of rationals of smaller cardinality that does not contain any perfect $6^{th}$ power but contains a $6^{th}$ power in $\mathbb{Q}_{p}$ for almost every $p$. The answer is yes. Consider the set $B = \{p_{1}, p_{2}, p_{1}p_{2}\}$ for any two distinct odd primes $p_{1}$ and $p_{2}$. $B$ contains a quadratic residue modulo every odd prime $p\notin\{p_{1}, p_{2}\}$ as a consequence of \[\bigg(\frac{p_{1}p_{2}}{p}\bigg) = \bigg(\frac{p_{1}}{p}\bigg)\bigg(\frac{p_{2}}{p}\bigg),\] where $\bigg(\frac{\cdot}{p}\bigg)$ is the Legendre symbol with respect to the prime $p$. Then, the Hensel's Lemma implies that $B$ contains a square in $\mathbb{Q}_{p}$ for almost every prime $p$. 

Moreover, the set $B^{(3)} = \{p_{1}^{3}, p_{2}^{3}, p_{1}^{3}p_{2}^{3}\}$ contains a $6^{th}$ power modulo $p$ for almost every prime $p$. This time, we have constructed infinitely many subsets of $\mathbb{Q}^{\times}/\big(\mathbb{Q}^{\times}\big)^{2}$ cardinality $3$ such that they do not contain any $6^{th}$ power but each of them contain a $6^{th}$ power in $\mathbb{Q}_{p}$ for almost every prime $p$. 

One may further ask if there exists a subset $A$ of rationals of cardinality $2$ such that $A$ does not contain a perfect square but $A$ contains a square in $\mathbb{Q}_{p}$ for almost every prime $p$. Does there exists a subset $A$ of three rationals that does not contain a perfect cube but contains a cube in $\mathbb{Q}_{p}$ for almost every $p$. The main result in this article (see Theorem \ref{GrunWang}) shows that the answer to both these questions is no. On the other hand, there does exist a subset $A$ of rationals with cardinality $2$ that contains a $6^{th}$ power in $\mathbb{Q}_{p}$ for almost every prime $p$ without containing a perfect $6^{th}$ power. However, such a set must necessarily be of the form \[\Big\{ -27 \cdot \alpha_{1}^{6}, \alpha_{2}^{2} \Big\} \text{  for distinct rationals } \alpha_{1}, \alpha_{2}.\]
Therefore, one may consider the following general problem. 
\begin{Problem}
Given a natural number $n \geq 2$, determine the largest integer $k(n)$ such that the following is satisfied:
\begin{center}
    A subset $A$ of rationals, of cardinality $\leq k(n)$, contains a $n^{th}$ power in $\mathbb{Q}_{p}$ for almost every prime $p$ if and only if either $(1)$ $A$ contains a $n^{th}$ power or $(2)$ elements of $A$ has a short and finite list of forms.
\end{center}
 \label{problem1}
\end{Problem}

Some partial answers of the problem exist in the literature. For instance, the following classical result, which was first obtained by Fried in \cite{Fr} and was later rediscovered by Filaseta and Richman in \cite{FiRi}, shows that $k(2) = 2$. 
\begin{proposition}
Let $A = \{a_{1}, a_{2}, \ldots, a_{\ell}\}$ be a finite set of non-zero integers. The following two conditions are equivalent:
\begin{enumerate}
    \item $A$ contains a square in $\mathbb{Z}_{p}$, for almost every prime $p$. 

    \item There exists a subset $B$ of $\{1, 2, \ldots, \ell\}$, of odd cardinality, such that $\prod_{j \in B} a_{j}$ is a perfect square. 
\end{enumerate}
\label{Fr}
\end{proposition}
A recent result of Farhangi and Magner shows that $k(n) \geq 3$ for odd $n$.

\begin{proposition}\big(see \cite[Theorem 2, pp. 17]{Sohail}\big) Let $\alpha, \beta, \gamma$ be three rational numbers and $n$ be any odd natural number greater than $1$. If none of $\alpha, \beta, \gamma$ are $n^{th}$ powers in $\mathbb{Q}$, then there exists infinitely many primes $p$ such that none of $\alpha, \beta, \gamma$ is an $n^{th}$ power modulo $p$.  
\end{proposition}
The main content of this article is the following theorem which shows that $k(n)$ is equal to the smallest prime dividing $n$, except in few outlined cases when $n$ is even. 

\begin{theorem}
Let $n = 2^{a_{0}} \cdot \prod_{i=1}^{k} p_{i}^{a_{i}}$ be a natural number with $a_{0},k \geq 0$ and $p_{1} < p_{2} < \ldots < p_{k}$ odd primes. Any finite subset $A$ of rationals with $|A| \leq \begin{cases} 2 ; \text{ when n is even } \\ p_{1}; \text{ when n is odd} \end{cases}$ contains a $n^{th}$ power in $\mathbb{Q}_{p}$ for almost every prime $p$ if and only if either $A$ contains a $n^{th}$ power or $n$ is even and $A$ has one of the following forms:
\begin{enumerate}
    \item If $a_{0} = 1$ and $n \neq 2$, then \[ A = \Big\{ \big((-1)^{\frac{p_{j}-1}{2}} \cdot p_{j}\big)^{\frac{n}{2}} \alpha_{1}^{n}, \alpha_{2}^{\frac{n}{p_{j}^{a_{j}}}}\Big\}, \text{ for some } 1 \leq j \leq k \text{ and } \alpha_{1}, \alpha_{2}\in\mathbb{Q}. \]

    \item If $a_{0} = 2$, then for some $1 \leq j \leq k$ and some $\alpha_{1}, \alpha_{2} \in\mathbb{Q}$,
    \begin{equation*}
        A \in \Bigg\{\bigg\{ -2^{n/2} \cdot \alpha_{1}^{n}, \alpha_{2}^{n/2}\bigg\}, \bigg\{ p_{j}^{n/2} \cdot \alpha_{1}^{n}, \alpha_{2}^{\bigg(\frac{n}{p_{j}^{a_{j}}}\bigg)} \bigg\}, \bigg\{p_{j}^{n/2} \cdot \alpha_{1}^{n}, -2^{\bigg(\frac{n}{2p_{j}^{a_{j}}}\bigg)} \cdot \alpha_{2}^{\bigg(\frac{n}{p_{j}^{a_{j}}}\bigg)} \bigg\} \Bigg\}
    \end{equation*}

    \item If $a_{0} \geq 3$, then one of the following holds:
    \begin{enumerate}
        \item $A$ contains an element of the form $2^{\frac{n}{2}} \cdot \alpha^{n}$ for some $\alpha\in\mathbb{Q}$.

        \item For some $1 \leq j \leq k$ and $\alpha_{1}, \alpha_{2} \in\mathbb{Q}$,
         \begin{multline*}
         A \in \Bigg\{ \bigg\{ p_{j}^{\frac{n}{2}} \cdot \alpha_{1}^{n}, \alpha_{2}^{\bigg(\frac{n}{p_{j}^{a_{j}}}\bigg)} \bigg\},  \bigg\{ p_{j}^{\frac{n}{2}} \cdot \alpha_{1}^{n}, 2^{\bigg(\frac{n}{2p_{j}^{a_{j}}}\bigg)} \cdot \alpha_{2}^{\bigg(\frac{n}{p_{j}^{a_{j}}}\bigg)} \bigg\}, \bigg\{ 2^{\frac{n}{2}} \cdot p_{j}^{\frac{n}{2}} \cdot \alpha_{1}^{n}, \alpha_{2}^{\bigg(\frac{n}{p_{j}^{a_{j}}}\bigg)} \bigg\} \\ \bigg\{ 2^{\frac{n}{2}} \cdot p_{j}^{\frac{n}{2}} \cdot \alpha_{1}^{n}, 2^{\bigg(\frac{n}{2p_{j}^{a_{j}}}\bigg)} \cdot \alpha_{2}^{\bigg(\frac{n}{p_{j}^{a_{j}}}\bigg)} \bigg\} \Bigg\}
         \end{multline*}
    \end{enumerate}
\end{enumerate}
Moreover, for every $n \geq 2$, there exists infinitely many subsets $A$ of $\mathbb{Q}^{\times}/\big(\mathbb{Q}^{\times}\big)^{p_{1}}$, with $|A| = \begin{cases} 3 ; \text{ for even } n \\ p_{1}+1 ; \text{ for odd } n\end{cases}$ such that $A$ contains $n^{th}$ power in $\mathbb{Q}_{p}$ for almost every $p$ but $A$ neither contains a $n^{th}$ power nor contains a subset from the above list.
\label{mainresult}
\end{theorem}
The Grunwald-Wang theorem for $n^{th}$ powers can be recovered by taking $|A| = 1$ in Theorem \ref{mainresult}. Just as Proposition \ref{GrunWang} is equivalent to the injectivity of the map \ref{injective}, our main result has a similar formulation. Let $n$ be a natural number greater than $1$, $q$ be the smallest prime dividing $n$ and $S$ be a finite subset of primes. Consider the map \begin{equation}
\Big(\mathbb{Q}^{\times}/\big(\mathbb{Q}^{\times}\big)^{n}\Big)^{q} \longrightarrow \prod_{p \notin S} \Bigg(\Big(\mathbb{Q}_{p}^{\times}/\big(\mathbb{Q}_{p}^{\times}\big)^{n}\Big)\Bigg)^{q} \label{map1},
\end{equation}
induced on each of the $q$ components by the natural map \eqref{injective}. Consider the subset \[\mathcal{S} := \Big\{ \big((a_{1p}, a_{2p}, \ldots, a_{qp})\big)_{p\notin S}: \text{ for every } p \notin S, \exists\hspace{1mm} 1 \leq j \leq q \text{ such that } a_{jp} = 1 \Big\}\] of \[\prod_{p \notin S} \Bigg(\Big(\mathbb{Q}_{p}^{\times}/\big(\mathbb{Q}_{p}^{\times}\big)^{n}\Big)\Bigg)^{q}\] and for every $1 \leq j \leq q$, define \[\mathcal{T}:= \Big\{(a_{1}, a_{2}, \ldots, a_{q}) : a_{j} = 1 \text{ for some } 1 \leq j \leq q\Big\}\] of $\Big(\mathbb{Q}^{\times}/\big(\mathbb{Q}^{\times}\big)^{n}\Big)^{q}$. Theorem \ref{mainresult} states that the pre-image of the set $\mathcal{S}$ under the map \eqref{map1} is exactly equal to $ \mathcal{T}$, except in some finitely many cases when $n$ is even. Furthermore, if one replaces $q$ by a larger natural number, preimage of $\mathcal{S}$ is always larger than $\mathcal{T}$ for every $n$. Theorem \ref{mainresult} when $n$ is an odd prime was first established in \cite{MishraFFA} (see Proposition \ref{FFA} below). 

This article has four more sections. The next section is used to collect some known results in preparation for the proof of Theorem \ref{mainresult}. Section \ref{oddn} is entirely dedicated to proving Theorem \ref{mainresult} for odd $n$, and the short Section \ref{evensection} does the same when $n$ is even. The last section shows that the upper bound $q$ in Theorem \ref{mainresult} is optimal for every $n$.

\begin{remark}
    In the statement of Theorem \ref{mainresult}, without loss of generality, we can assume that $A$ is a subset of integers that contains a $n^{th}$ power in $\mathbb{Z}_{p}$ for almost every $p$. This is because if \[A := \Big\{\frac{a_{1}}{b_{1}}, \frac{a_{2}}{b_{2}}, \ldots, \frac{a_{k}}{b_{k}}\Big\}\] is a subset of rationals that contain a $n^{th}$ power in $\mathbb{Q}_{p}$ for almost every prime $p$, then \[A^{\prime} := \Bigg\{ \frac{a_{1}}{b_{1}} \big(b_{1}b_{2} \cdots b_{k}\big)^{n}, \frac{a_{2}}{b_{2}} \big(b_{1}b_{2} \cdots b_{k}\big)^{n}, \ldots, \frac{a_{k}}{b_{k}} \big(b_{1}b_{2} \cdots b_{k}\big)^{n}\Bigg\}\] also does the same in $\mathbb{Z}_{p}$ for almost every prime $p$. Therefore, we will assume without loss of generality that $A$ is a finite subset of integers containing a $n^{th}$ power in $\mathbb{Z}_{p}$ for almost every prime $p$ - which is equivalent to the fact that $A$ contains a $n^{th}$ power modulo $p$, for almost every prime $p$.
\end{remark}

\section{Some Preliminary Results}
Schinzel and  Ska\l{}ba, in \cite{SS}, obtained a necessary and sufficient condition for a finite subset of a number field $K$ to contain a $n^{th}$ power in $K_{p}$ for almost every prime $p$ of $K$. This condition was further refined by Ska\l{}ba when $n$ is a power of an odd prime. We will state and use this refined result, albeit only in the context of rationals. 
\begin{proposition}\big(see \cite[Theorem 2, pp. 68]{Sk}\big)
Let $q$ be an odd prime, $m \geq 1$ and $A = \{a_{j}\}_{j=1}^{\ell}$ be a finite subset of non-zero rationals. Then, the following two conditions are equivalent:
\begin{enumerate}
    \item $A$ contains a $q^{m}$-th power in $\mathbb{Q}_{p}$ for almost every prime $p$. 

    \item For every $(c_{j})_{j=1}^{\ell}\in\mathbb{Z}^{\ell}$, there exists subsets $B, C  \subset\{1, 2, \ldots, \ell\}$ with $|B| \not\equiv |C| \hspace{1mm} (\text{mod } q)$ such that \[\frac{\prod_{j \in B} a_{j}^{c_{j}}}{\prod_{j\in C} a_{j}^{c_{j}}} \text{ is equal to a perfect } q^{m}-\mathrm{th} \text{ power in rationals}.\]
\end{enumerate}
\label{Skalba}
\end{proposition}
We note that the subsets $B$ or $C$ in the Proposition \ref{Skalba} above are allowed to be empty, in which case, the product of $a_{j}^{c_{j}}$ over the empty set is taken to be $1$.

We will also need a characterization of subsets $A$ which contain a fixed prime power $q$ in $\mathbb{Z}_{p}$, for almost every prime $p$. This characterization should ideally allow one to construct and recognize such sets $A$ with relative ease. Before stating such a characterization, we need to associate every finite subset of $q$-free integers with a collection of linear hyperplanes in some vector space over $\mathbb{F}_{q}$.

\subsection{Linear Hyperplanes Associated to a Finite Subset}
Let $q$ be an odd prime and $A = \{a_{j}\}_{j=1}^{\ell}$ be a finite subset of $q$-free integers, neither containing $1$ nor $-1$. Let $p_{1}, p_{2}, \ldots, p_{s}$ be all the primes dividing $\prod_{j=1}^{\ell} a_{j}$. Furthermore, for every $1 \leq i \leq s$ and $1 \leq j \leq \ell$, let $\mu_{ij} \geq 0$ be such that $p_{i}^{\mu_{ij}} \mid\mid a_{j}$. Define 
\[\mathcal{H}_{j} := \Big\{ (x_{i})_{i=1}^{s} \in\big(\mathbb{F}_{q}\big)^{s} : \sum_{i=1}^{s} \mu_{ij} x_{i} = 0 \Big\}.\] Then, the set $\big\{\mathcal{H}_{j}\big\}_{j=1}^{\ell}$ is said to be the set of hyperplanes (in $\mathbb{F}_{q}^{s})$ associated to the set $A$ of integers. The author proved the following equivalence in \cite{MishraFFA}.

\begin{proposition}(see \cite[Theorem 1]{MishraFFA})
Let $q$ be an odd prime and $A = \{a_{j}\}_{j=1}^{\ell}$ be a finite subset of $q$-free integers, neither containing $1$ nor $-1$. Then, the following two conditions are equivalent:
\begin{enumerate}
    \item The set $A$ contains a $q^{th}$ power in $\mathbb{Z}_{p}$ for almost every prime $p$. 

    \item The set of hyperplanes \[\Bigg\{ \Big\{ (x_{i})_{i=1}^{s} \in\big(\mathbb{F}_{q}\big)^{s} : \sum_{i=1}^{s} \mu_{ij} x_{i} = 0 \Big\} :  1 \leq j \leq \ell \Bigg\} \] covers $\mathbb{F}_{q}^{s}$, i.e., every $(y_{i})_{i=1}^{s}\in\mathbb{F}_{q}^{s}$ satisfies $\sum_{i=1}^{s} \mu_{ij} y_{i} = 0$ in $\mathbb{F}_{q}$ for some $1 \leq j \leq \ell$. 
\end{enumerate}
\label{thmFFA}
\end{proposition}

A consequence of the above Proposition is the following result that $k(q) \geq q$. 
\begin{proposition}\big(See {\cite[Corollary 1]{MishraFFA}}\big) Let $q$ be an odd prime and let $A$ be a subset of integers not containing any perfect $q^{th}$ power but containing a $q^{th}$ power in $\mathbb{Q}_{p}$ for almost every prime $p$. Then $|A| \geq q+1$. \label{FFA}
\end{proposition}

\section{When $n$ is odd}\label{oddn}
Let $q$ be an odd prime. Assume that a subset $A$ of integers contains a $q^{m}$-th power in $\mathbb{Z}_{p}$ for almost every prime $p$, without containing a perfect $q^{m}$-th power. Proposition \ref{FFA} cannot be directly used to conclude that the cardinality of $A$ is at least $q+1$. This is because $A$ may still contain perfect $q^{r}$-th powers, for some $r < m$. The next lemma upgrades Proposition \ref{FFA} so that we can obtain this conclusion.

\begin{lemma}
Let $q$ be an odd prime, $m \geq 2$ and $C = \{a_{1}^{q^{\mu_{1}}}, a_{2}^{q^{\mu_{2}}}, \ldots, a_{\nu}^{q^{\mu_{\nu}}}, a_{\nu + 1}, \ldots, a_{\ell}\}$ be a finite subset of integers, where $1 \leq \mu_{i} \leq m-1$ and $a_{j}$ are $q$-free integers for every $\nu + 1 \leq i \leq \ell$. If $C$ contains a $q^{m}$-th power in $\mathbb{Z}_{p}$, for almost every prime $p$, then the set $A = \{a_{1}, a_{2}, \ldots, a_{\nu}, a_{\nu + 1}, \ldots, a_{\ell}\}$ contains a $q$-th power in $\mathbb{Z}_{p}$, for almost every $p$. \label{qfree}
\end{lemma}

\begin{proof}
Let $C = \{c_{j}\}_{j=1}^{\ell}$, where $c_{j} = \begin{cases} a_{j}^{q^{\mu_{j}}} ; \text{ for } 1 \leq j \leq \nu \\ a_{j} ; \text{ for } \nu+1 \leq j \leq \ell. \end{cases}$ For the sake of contradiction, assume that $A$ fails to contain $q^{th}$ power in $\mathbb{Z}_{p}$ for infinitely many primes $p$. Then, Proposition \ref{Skalba} implies:
\begin{multline}
       \\ \text{ There exists } (d_{j})_{j=1}^{\ell} \in\mathbb{Z}^{\ell} \text{ such that, for every } E, F \subset\{1, 2, \ldots, \ell\} \\ \text{ with } |E| \not\equiv |F| \hspace{1mm}(\text{mod } q), \text{ we have that } \frac{\prod_{j \in E} a_{j}^{d_{j}}}{\prod_{j \in F} a_{j}^{d_{j}}} \text{ is not a } q^{th} \text{ power. } \label{intermediate} \\
\end{multline}

Define $a := \mathrm{max}_{1 \leq j \leq \nu}\{ \mu_{j} \}$ which is at most $m-1$ and \[(e_{j})_{j=1}^{\ell}\in\mathbb{Z}^{\ell} \text{ by setting } e_{j} = \begin{cases} q^{a-\mu_{j}} \cdot  d_{j} ; \text{ for } 1 \leq j \leq \nu \\ q^{a} d_{j} ; \text{ for } \nu + 1 \leq j \leq \ell. \end{cases} \]
Then, for any $E, F \subset\{1, 2, \ldots, \ell\}$ with $|E| \not\equiv |F|\hspace{1mm} (\text{mod } q)$, we have that 
\begin{multline*}
  \\ \frac{\prod_{j \in E} c_{j}^{e_{j}}}{\prod_{j \in F} c_{j}^{e_{j}}} \\\\ = \frac{\prod_{j \in E, 1 \leq j \leq \nu} c_{j}^{e_{j}}}{\prod_{j \in F, 1 \leq j \leq \nu} c_{j}^{e_{j}}} \times \frac{\prod_{j \in E, \nu + 1 \leq j \leq \ell} c_{j}^{e_{j}}}{\prod_{j \in F, \nu + 1 \leq j\leq \ell} c_{j}^{e_{j}}} \\\\ =  \frac{\prod_{j \in E, 1 \leq j \leq \nu} \big( a_{j}^{q^{\mu_{j}}}\big)^{d_{j} \cdot q^{a-\mu_{j}}}}{\prod_{j \in F, 1 \leq j \leq \nu} \big( a_{j}^{q^{\mu_{j}}}\big)^{d_{j} \cdot q^{a-\mu_{j}}}} \times \frac{\prod_{j \in E, \nu + 1 \leq j \leq \ell} \big(a_{j}\big)^{q^{a} \cdot d_{j}}}{\prod_{j \in F, \nu + 1 \leq j\leq \ell} \big(a_{j}\big)^{q^{a} \cdot d_{j}}} \\
\end{multline*}
\begin{multline*}
    \\  = \frac{\prod_{j \in E, 1 \leq j \leq \nu} \big( a_{j}\big)^{d_{j} \cdot q^{a}}}{\prod_{j \in F, 1 \leq j \leq \nu} \big( a_{j}\big)^{d_{j} \cdot q^{a}}} \times \frac{\prod_{j \in E, \nu + 1 \leq j \leq \ell} \big(a_{j}\big)^{q^{a} \cdot d_{j}}}{\prod_{j \in F, \nu + 1 \leq j\leq \ell} \big(c_{j}\big)^{q^{a} \cdot d_{j}}} = \Bigg( \frac{\prod_{j\in E} a_{j}^{d_{j}}}{\prod_{j \in F} a_{j}^{d_{j}}} \Bigg)^{q^{a}}.\\
\end{multline*}
Since \[\frac{\prod_{j\in E} a_{j}^{d_{j}}}{\prod_{j \in F} a_{j}^{d_{j}}}\] is never a perfect $q^{th}$ power for any $|E| \not\equiv |F| \hspace{1mm} (\text{mod } q)$, and since $a \leq m-1$, we have that \[\frac{\prod_{j \in E} c_{j}^{e_{j}}}{\prod_{j \in F} c_{j}^{e_{j}}} = \Bigg( \frac{\prod_{j\in E} a_{j}^{d_{j}}}{\prod_{j \in F} a_{j}^{d_{j}}} \Bigg)^{q^{a}}\] is not a  perfect $q^{m}$-th power, for any $|E| \not\equiv |F| \hspace{1mm} (\text{mod } q)$. Here, we are using the fact that for an odd prime $q$ and rational integer $x$, if $x$ is not a perfect $q^{th}$ power then $x^{q^{m-1}}$ is not a perfect $q^{m}$-th power. 

To summarize, we have an $(e_{j})_{j=1}^{\ell}\in\mathbb{Z}^{\ell}$ such that for any $E, F \subset\{1, 2, \ldots, \ell\}$ with $|E| \not\equiv |F| (\text{mod } q)$, \[ \frac{\prod_{j \in E} c_{j}^{e_{j}}}{\prod_{j \in F} c_{j}^{e_{j}}}\] is not a $q^{m}$-th power. Therefore, as a consequence of Proposition \ref{Skalba}, the set $C$ does not contain a $q^{m}$-th power in $\mathbb{Z}_{p}$, for infinitely many primes $p$. 
\end{proof}

Using the previous lemma, we will upgrade Proposition \ref{FFA} to subsets $A$ that contain a $q^{m}$-th power in $\mathbb{Z}_{p}$ for almost every $p$ but do not contain a perfect $q^{m}$-th power. 
\begin{lemma}
Let $q$ be an odd prime, and let $A = \{a_{j}\}_{j=1}^{\ell}$ be a finite subset of integers not containing any perfect $q^{m}$-th power but containing a $q^{m}$-th power in $\mathbb{Z}_{p}$, for almost every prime $p$. Then, $|A| \geq q+1$. 
\label{primepower}
\end{lemma}

\begin{proof}
If $A$ did not contain any perfect $q^{th}$ power, the lemma would follow from Proposition \ref{FFA}. Therefore, assume that upon re-indexing elements of $A$, $a_{1}, a_{2}, \ldots, a_{\nu}$ of them are each perfect $q^{\mu_{j}}$-th powers, for some $1 \leq \mu_{j} \leq m-1$ respectively.

Since $A$ does not contain a perfect $q^{m}$-th power, for every $1 \leq j \leq \nu$, there exists a largest natural number $\mu_{\nu} \leq m-1$ such that $a_{\nu} = (b_{\nu})^{q^{\mu_{\nu}}}$ and $b_{\nu}$ is not a perfect $q^{th}$ power.  Using Lemma \ref{qfree}, we have that the set $B = \{b_{1}, b_{2}, \ldots, b_{\nu}, a_{\nu+1}, \ldots, a_{\ell}\}$ then contains a $q^{th}$ power in $\mathbb{Z}_{p}$, for almost every prime $p$. Since $B$ does not contain a $q^{th}$ power anymore, Proposition \ref{FFA} implies that $|B| = \ell = |A| \geq q + 1$.
\end{proof}

To prove Theorem \ref{mainresult} for general odd numbers $n$, we will need the following lemma, which shows that the property of subsets of integers to contain a $q^{m}$-th power in $\mathbb{Z}_{p}$, for almost every $p$, is invariant under exponentiation by tuples with entries in $\mathbb{F}_{q}^{\times}$ in the following sense. For $m = 1$, the lemma was first proved by Mishra in \cite{MishraFFA}.

\begin{lemma}
Let $q$ be an odd prime and $m$ be a natural number. Let $A = \{a_{j}\}_{j=1}^{\ell}$ be a finite subset of integers and $(c_{j})_{j=1}^{\ell} \in \mathbb{Z}^{\ell}$ such that $c_{j} \not\equiv 0 \hspace{1mm} (\text{mod } q)$ for every $1 \leq j \leq \ell$. Then $A$ contains a $q^{m}-\mathrm{th}$ power in $\mathbb{Q}_{p}$ for almost every $p$ if and only if $A^{\Vec{c}}=\{a_{j}^{c_{j}}\}_{j=1}^{\ell}$ does the same. \label{exponentiation2}
\end{lemma}

\begin{proof}
Since $c_{j} \not\equiv 0 \hspace{1mm} (\text{mod } q)$ for every $1 \leq j \leq \ell$, there exists $(d_{j})_{j=1}^{\ell}\in\mathbb{Z}^{\ell}$ such that for every $1 \leq j \leq \ell$, $c_{j}d_{j} \equiv 1 \hspace{1mm} (\text{mod } q^{m})$, i.e., $c_{j}d_{j} = 1 + k_{j} \cdot q^{m}$ for some $k_{j}\in\mathbb{Z}$. For every $E, F \subset\{1, 2, \ldots, \ell\}$ and integers $\alpha_{1}, \alpha_{2}, \ldots, \alpha_{\ell}$, we have that

\begin{multline}
    \\ \frac{\prod_{j\in E}(\alpha_{j}^{c_{j}})^{d_{j}}}{\prod_{j\in F} (\alpha_{j}^{c_{j}})^{d_{j}}} = \frac{\prod_{j\in E}\alpha_{j}^{1 + k_{j}q^{m}}}{\prod_{j\in F} \alpha_{j}^{1 + k_{j}q^{m}}} = \Big( \frac{\prod_{j\in E}\alpha_{j}^{k_{j}}}{\prod_{j\in F} \alpha_{j}^{k_{j}}} \Big)^{q^{m}} \times \frac{\prod_{j\in E}\alpha_{j}}{\prod_{j\in F} \alpha_{j}}. \label{generalexpo}\\
\end{multline}
On the one hand, if $A$ does not contain a $q^{m}$-th power in $\mathbb{Z}_{p}$ for infinitely many primes $p$, then by Proposition \ref{Skalba}, there exists $(b_{j})_{j=1}^{\ell} \in\mathbb{Z}^{\ell}$ such that for any two subsets $E, F$ of $\{1, 2, \ldots, \ell\}$ with $|E| \not\equiv |F| \hspace{1mm} (\text{mod } q)$, \[\frac{\prod_{j\in E}a_{j}^{b_{j}}}{\prod_{j\in F} a_{j}^{b_{j}}}\] is not equal to a $q^{m}$-th power. In this case, replacing $\alpha_{j}$'s by $a_{j}^{b_{j}}$'s in \eqref{generalexpo} shows that \[\frac{\prod_{j\in E}(a_{j}^{c_{j}})^{b_{j} d_{j}}}{\prod_{j\in F} (a_{j}^{c_{j}})^{b_{j} d_{j}}}\] is not a $q^{m}$-th power either. In other words, there exists $(b_{j}d_{j})_{j=1}^{\ell}\in\mathbb{Z}^{\ell}$ such that for every $E, F \subset\{1, 2, \ldots, \ell\}$, with $|E| \not\equiv |F| \hspace{1mm} (\text{mod } q)$, we have that \[\frac{\prod_{j\in E}(a_{j}^{c_{j}})^{b_{j} d_{j}}}{\prod_{j\in F} (a_{j}^{c_{j}})^{b_{j} d_{j}}}\] is not a $q^{m}$-th power. So, Proposition \ref{Skalba} implies that $A^{\Vec{c}}$ does not contain a $q^{m}$-th power in $\mathbb{Z}_{p}$, for infinitely many primes $p$.

On the other hand, if $A^{\Vec{c}}$ does not contain a $q^{m}$-th power in $\mathbb{Z}_{p}$ for infinitely many primes $p$, then Proposition \ref{Skalba} again implies that there is a $(e_{j})_{j=1}^{\ell}\in\mathbb{Z}^{\ell}$ such that \[\frac{\prod_{j\in E}(a_{j}^{c_{j}})^{e_{j}}}{\prod_{j\in F} (a_{j}^{c_{j}})^{e_{j}}} = \frac{\prod_{j\in E}(a_{j})^{c_{j}e_{j}}}{\prod_{j\in F}(\alpha_{j})^{c_{j} e_{j}}}\] is not a $q^{m}$-th power for any $|E| \not\equiv |F| \hspace{1mm} (\text{mod } q)$. Then, Proposition \ref{Skalba} at once implies that $A$ does not contain a $q^{m}$-th power in $\mathbb{Z}_{p}$, for a positive density of primes.
\end{proof}

Now, we are ready to prove Theorem \ref{mainresult} for a general odd natural number. 
\subsection{Proof of Theorem \ref{mainresult}, when $n$ is odd:}
If $A$ does not contain a perfect $p_{i}^{th}$ power for some $1 \leq i \leq k$, then Proposition \ref{FFA} implies $|A| \geq p_{i}+1 \geq p_{1} + 1$. On the other hand, if $A$ does not contain a perfect $p_{i}^{a_{i}}$-th power, then Lemma \ref{primepower} gives us that $|A| \geq p_{i} + 1 \geq p_{1} + 1$. However, as noted before, $A$ may contain a $p_{i}^{a_{i}}$-th power for every $1 \leq i \leq k$.

If such is the case, since $A$ does not contain a perfect $n^{th}$ power, there exists a divisor $m$ of $n$ such that the following holds:
\begin{itemize}
    \item $m = \prod_{i=1}^{\mu} p_{i}^{c_{i}}$ and there exists $1 \leq i_{0} \leq \mu$ such that $c_{i_{0}} < a_{i_{0}}$, upon reordering of $p_{i}$s if necessary.

    \item If $A$ contains a perfect $m^{\prime}$-th power for any other divisor $m^{\prime}$ of $n$, then the largest power of $p_{i_{0}}$ that divides $m^{\prime}$ is $\leq p_{i_{0}}^{c_{i_{0}}}$. 
\end{itemize}

Let us assume that $A = \{a_{1}^{m}, a_{2}^{m}, \ldots, a_{\nu_{1}}^{m}, a_{\nu_{1}+1}, \ldots, a_{\ell}\}$ for some $\nu_{1} \geq 1$, upon reindexing elements of $A$ if necessary. For each $1 \leq j \leq \nu_{1}$, we have \[a_{j}^{m} = \big( a_{j}^{p_{i_{0}}^{c_{i_{0}}}} \big)^{\prod_{i=1, i \neq i_{0}}^{\mu} p_{i}^{c_{i}}}.\]
Furthermore, for each $\nu_{1} + 1 \leq j \leq \ell$, \[ a_{j} = \Big( b_{j}^{p_{i_{0}}^{d_{j}}}\Big)^{e_{j}},\]
for some $0 \leq d_{j} \leq c_{i_{0}} < a_{i_{0}}$, $(e_{j}, p_{i_{0}}) = 1$ and $p_{i_{0}}$-th power free $b_{j}$. Since $A$ contains a $p_{i_{0}}^{a_{i_{0}}}$-th power in $\mathbb{Z}_{p}$ for almost every prime $p$, Proposition \ref{exponentiation2} implies that the set \[ A^{\prime} := \Bigg\{a_{1}^{p_{i_{0}}^{c_{i_{0}}}}, a_{2}^{p_{i_{0}}^{c_{i_{0}}}}, \ldots, a_{\nu_{1}}^{p_{i_{0}}^{c_{i_{0}}}}, b_{\nu_{1}+1}^{p_{i_{0}}^{d_{\nu_{1}+1}}}, \ldots, b_{\ell}^{p_{i_{0}}^{d_{\ell}}}\Bigg\}\] still contains a $p_{i_{0}}^{a_{i_{0}}}$-th power in $\mathbb{Z}_{p}$, for almost every prime $p$. Note that the set $A^{\prime}$ does not contain a perfect $p_{i_{0}}^{a_{i_{0}}}$-th power. Therefore, using Lemma \ref{primepower}, we have that $|A^{\prime}| = \ell = |A| \geq p_{i_{0}} + 1 \geq p_{1} + 1$. $\square$

\section{When $n$ is Even.}\label{evensection}
Theorem \ref{mainresult} for even $n$ is a direct consequence of the following result obtained by Schinzel and Ska\l{}ba, for two element subsets of a general number field. Again, we only state the result in the context of rational numbers.

\begin{proposition}\big(See {\cite[Corollary 4]{SS}}\big)
Let $n = 2^{a_{0}} \prod_{i=1}^{k} p_{i}^{a_{i}}$ be a natural number, where $ a_{0} \geq 0$ and $p_{1} < p_{2} < \ldots < p_{k}$ are odd primes, and let $A$ be a $2$-element subset of rationals. Then $A$ contains a $n^{th}$ power in $\mathbb{Q}_{p}$, for almost every prime $p$ if and only if either $A$ contains a $n^{th}$ power or one of the following holds:
\begin{enumerate}
    \item $a_{0} = 1$, $n \neq 2$ and \[ A = \Big\{ \big((-1)^{\frac{p_{j}-1}{2}} \cdot p_{j}\big)^{\frac{n}{2}} \alpha_{1}^{n}, \alpha_{2}^{\frac{n}{p_{j}^{a_{j}}}}\Big\}, \text{ for some } 1 \leq j \leq k \text{ and } \gamma_{1}, \gamma_{2}\in\mathbb{Q}. \]
    
    \item $a_{0} = 2$ and for some $1 \leq j \leq k$ and $\alpha_{1}, \alpha_{2} \in\mathbb{Q}$,
    \begin{multline*}
        A \in \Bigg\{ \bigg\{ -2^{n/2} \cdot \alpha_{1}^{n}, \alpha_{2}^{n/2}\bigg\}, \bigg\{ p_{j}^{n/2} \cdot \alpha_{1}^{n}, \alpha_{2}^{\bigg(\frac{n}{p_{j}^{a_{j}}}\bigg)} \bigg\}, \bigg\{p_{j}^{n/2} \cdot \alpha_{1}^{n}, -2^{\bigg(\frac{n}{2p_{j}^{a_{j}}}\bigg)} \cdot \alpha_{2}^{\bigg(\frac{n}{p_{j}^{a_{j}}}\bigg)} \bigg\} \Bigg\}
    \end{multline*}

    \item $a_{0} \geq 3$ and one of the following holds:
    \begin{enumerate}
        \item $A$ contains an element of the form $2^{\frac{n}{2}} \cdot \alpha^{n}$ for some $\alpha\in\mathbb{Q}$.

        \item For some $1 \leq j \leq k$ and $\alpha_{1}, \alpha_{2} \in\mathbb{Q}$,
         \begin{multline*}
         A \in \Bigg\{ \bigg\{ p_{j}^{\frac{n}{2}} \cdot \alpha_{1}^{n}, \alpha_{2}^{\bigg(\frac{n}{p_{j}^{a_{j}}}\bigg)} \bigg\},  \bigg\{ p_{j}^{\frac{n}{2}} \cdot \alpha_{1}^{n}, 2^{\bigg(\frac{n}{2p_{j}^{a_{j}}}\bigg)} \cdot \alpha_{2}^{\bigg(\frac{n}{p_{j}^{a_{j}}}\bigg)} \bigg\}, \bigg\{ 2^{\frac{n}{2}} \cdot p_{j}^{\frac{n}{2}} \cdot \alpha_{1}^{n}, \alpha_{2}^{\bigg(\frac{n}{p_{j}^{a_{j}}}\bigg)} \bigg\} \\ \bigg\{ 2^{\frac{n}{2}} \cdot p_{j}^{\frac{n}{2}} \cdot \alpha_{1}^{n}, 2^{\bigg(\frac{n}{2p_{j}^{a_{j}}}\bigg)} \cdot \alpha_{2}^{\bigg(\frac{n}{p_{j}^{a_{j}}}\bigg)} \bigg\} \Bigg\}
         \end{multline*}
    \end{enumerate}
\end{enumerate}
\label{even-hammer}
\end{proposition}

\section{Optimality of the lower bound in Theorem \ref{mainresult}}\label{Optimality}
In this section, we want to show that the lower bound obtained in Theorem \ref{mainresult} is optimal. In other words, we will show that if $p_{1}$ is the smallest prime dividing $n$ and $p_{1}^{a_{1}} \mid\mid n$, then there are infinitely many subsets of $\mathbb{Q}/\big(\mathbb{Q}^{\times}\big)^{p_{1}^{a_{1}}}$ with $\geq p_{1}+1$ elements that does not contain a $n^{th}$ power. When $n$ is even, we also want to show that these sets of cardinality $3$ do not properly contain sets listed in Proposition \ref{even-hammer}.

Note that for every pair of distinct primes $q_{1}, q_{2}$ different from $p_{1}$, the set \[\big\{q_{1}, q_{2}, q_{1}q_{2}, q_{1}q_{2}^{2}, \ldots, q_{1}q_{2}^{p_{1}-1}\big\}\] contains a $p_{1}^{th}$ power in $\mathbb{Z}_{p}$ for almost every prime $p$. When $p_{1}$ is odd, this is is a consequence of Proposition \ref{thmFFA} and the observation that the set of hyperplanes \[\big\{(x_{1}, x_{2}): x_{2} = 0 \text{ or } x_{1} + c \cdot x_{2} = 0 \text{ for some } c\in\{0, 1, \ldots, q-1\} \big\}\] covers $\mathbb{F}_{q}^{2}$. When $p_{1} = 2$, this is the consequence of the fact that the product $$ q_{1} \times q_{2} \times q_{1}q_{2}$$ is a perfect square. 

Recall that $n$ is a natural number with prime factorization $n = 2^{a_{0}} \prod_{i=1}^{k} p_{k}^{a_{k}}$, where $a_{0} \geq 0$ and $a_{i} \geq 1$. Now, we will show the optimality of the Theorem \ref{mainresult} according to the parity of $n$. 
\begin{enumerate}
    \item When $n$ is even, choose $2$ odd primes, $q_{1}$ and $q_{2}$. The set \[\{q_{1}, q_{2}, q_{1}q_{2}\}\] contains a square in $\mathbb{Z}_{p}$, for almost every $p$, as a consequence of Proposition \ref{Fr}. Then, the set \[\Big\{\big(q_{1}\big)^{2^{a_{0}-1} \cdot \prod_{i=1}^{k} p_{i}^{a_{i}}}, \big(q_{2}\big)^{2^{a_{0}-1} \cdot \prod_{i=1}^{k} p_{i}^{a_{i}}}, \big(q_{1}q_{2}\big)^{2^{a_{0}-1} \cdot \prod_{i=1}^{k} p_{i}^{a_{i}}} \Big\}\] contains a $n^{th}$ power in $\mathbb{Z}_{p}$, for almost every prime $p$. 

    \item When $n$ is odd, i.e., $a_{0} = 0$, choose $2$ odd primes $q_{1}$ and $q_{2}$. The set \[\Big\{q_{1}, q_{2}, q_{1}q_{2}, q_{1}q_{2}^{2}, \ldots, q_{1}q_{2}^{p-1} \Big\}\] contains a $p_{1}^{th}$ power in $\mathbb{Z}_{p}$, for almost every prime $p$, as discussed above. So, the set \[\Big\{q_{1}^{p_{1}^{a_{1}-1} \cdot \prod_{i=2}^{k} p_{i}^{a_{i}}}, q_{2}^{p_{1}^{a_{1}-1} \cdot \prod_{i=2}^{k} p_{i}^{a_{i}}}, \big(q_{1}q_{2}\big)^{p_{1}^{a_{1}-1} \cdot \prod_{i=2}^{k} p_{i}^{a_{i}}}, \ldots, \big(q_{1}q_{2}^{p-1}\big)^{p_{1}^{a_{1}-1} \cdot \prod_{i=2}^{k} p_{i}^{a_{i}}} \Big\}\] contains a $n^{th}$ power in $\mathbb{Z}_{p}$, for almost every prime $p$. 
\end{enumerate} 
Note that the sets in $(1)$ and $(2)$ do not contain perfect $p_{1}^{a_{1}}$-th power and that $q_{1}, q_{2}$ are any two distinct primes different from $p_{1}$. Therefore, we obtain infinitely many subsets of $\mathbb{Q}^{\times}/\big(\mathbb{Q}^{\times}\big)^{p_{1}^{a_{1}}}$ that have cardinality $(p_{1} + 1)$, contain a $n^{th}$ power in $\mathbb{Q}_{p}$ for almost every $p$, but do not contain any perfect $n^{th}$ power. 

\bibliographystyle{plain}
\bibliography{local_global}

\begin{thebibliography}{10}

\bibitem{ArtinTate}
Emil Artin and John Tate.
\newblock {\em Class field theory}.
\newblock AMS Chelsea Publishing, Providence, RI, 2009.
\newblock Reprinted with corrections from the 1967 original.

\bibitem{Creutz}
Brendan Creutz.
\newblock A {G}runwald-{W}ang type theorem for abelian varieties.
\newblock {\em Acta Arith.}, 154(4):353--370, 2012.

\bibitem{Sohail}
Sohail Farhangi and Richard Magner.
\newblock On the partition regularity of {$ax+by=cw^mz^n$}.
\newblock {\em Integers}, 23:Paper No. A18, 52, 2023.

\bibitem{FiRi}
Michael~A. Filaseta and David~R. Richman.
\newblock Sets which contain a quadratic residue modulo {$p$} for almost all {$p$}.
\newblock {\em Math. J. Okayama Univ.}, 31:1--8, 1989.

\bibitem{Fr}
Michael Fried.
\newblock {Arithmetical properties of value sets of polynomials}.
\newblock {\em Acta Arith.}, 15:91--115, 1969.

\bibitem{Grunwald}
Wilhelm Grunwald.
\newblock Ein allgemeines {E}xistenztheorem f\"{u}r algebraische {Z}ahlk\"{o}rper.
\newblock {\em J. Reine Angew. Math.}, 169:103--107, 1933.

\bibitem{MishraFFA}
Bhawesh Mishra.
\newblock {Prime power residues and linear coverings of vector space over $\mathbb{F}_{q}$}.
\newblock {\em Finite Fields Appl.}, 89, 2023.

\bibitem{NeukrichCohom}
J\"{u}rgen Neukirch, Alexander Schmidt, and Kay Wingberg.
\newblock {\em Cohomology of number fields}, volume 323 of {\em Grundlehren der mathematischen Wissenschaften [Fundamental Principles of Mathematical Sciences]}.
\newblock Springer-Verlag, Berlin, second edition, 2008.

\bibitem{Roquette}
Peter Roquette.
\newblock {\em The {B}rauer-{H}asse-{N}oether theorem in historical perspective}, volume~15 of {\em Schriften der Mathematisch-Naturwissenschaftlichen Klasse der Heidelberger Akademie der Wissenschaften [Publications of the Mathematics and Natural Sciences Section of Heidelberg Academy of Sciences]}.
\newblock Springer-Verlag, Berlin, 2005.

\bibitem{SS}
Andrzej Schinzel and Mariusz Ska\l{}ba.
\newblock {On power residues}.
\newblock {\em Acta Arith.}, 108(1):77--94, 2003.

\bibitem{Sk}
Mariusz Ska\l{}ba.
\newblock On sets which contain a $q^{th}$ power residue for almost all prime modulus.
\newblock {\em Colloq. Math.}, 102(1):67--71, 2005.

\bibitem{Wang1}
Shianghaw Wang.
\newblock A counter-example to {G}runwald's theorem.
\newblock {\em Ann. of Math. (2)}, 49:1008--1009, 1948.

\bibitem{Wang2}
Shianghaw Wang.
\newblock On {G}runwald's theorem.
\newblock {\em Ann. of Math. (2)}, 51:471--484, 1950.

\bibitem{Effective}
Song Wang.
\newblock Grunwald-{W}ang theorem, an effective version.
\newblock {\em Sci. China Math.}, 58(8):1589--1606, 2015.

\bibitem{Whaples}
George Whaples.
\newblock Non-analytic class field theory and {G}r\"{u}nwald's theorem.
\newblock {\em Duke Math. J.}, 9:455--473, 1942.

\end{thebibliography}
\end{document}